\documentclass[reqno,11pt]{article}

\usepackage{geometry}
\usepackage{indentfirst}
\usepackage{amsthm}
\usepackage[numbers,comma,sort&compress,numbers]{natbib}
\usepackage{amsmath,amsfonts,amssymb,color}
\numberwithin{equation}{section}
\usepackage{enumerate}
\usepackage{fixltx2e}
\usepackage{mathtools}
\usepackage[T1]{fontenc}
\usepackage[utf8]{inputenc}
\usepackage{authblk}
\allowdisplaybreaks[4]

\theoremstyle{plain}
\newtheorem{theorem}{Theorem}[section]

\newtheorem{lemma}[theorem]{Lemma}
\newtheorem{proposition}[theorem]{Proposition}

\theoremstyle{definition}
\newtheorem{definition}[theorem]{Definition}

\newtheorem{remark}[theorem]{Remark}

\title{\textbf{Normalized solutions for Schr\"{o}dinger-Bopp-Podolsky system}}

\author{, Lin Li and }
\author{Chuan-Min He, \ Lin Li\footnote{{\tt Corresponding author.}
		E-mail address: {\tt linli@ctbu.edu.cn \& lilin420@gmail.com} (L. Li).} \ and Shang-Jie Chen\footnote{This work is supported by Research Fund of National Natural Science Foundation of China (No. 11861046), National Natural Science Foundation of Chongqing(No. 2019jcyj-msxmX0115), Chongqing Municipal Education Commission (No. KJQN20190081), Chongqing Technology and Business University(No. CTBUZDPTTD201909).}\\
	\footnotesize
	School of Mathematics and Statistics \& Chongqing Key Laboratory of Economic and Social Application Statistics,\\ \footnotesize Chongqing Technology and Business University,\\
	\footnotesize Chongqing 400067, China}

\date{}

\begin{document}
\maketitle

\begin{abstract}
\noindent In this paper, we study the following energy functional originates from the Schr\"{o}dinger-Bopp-Podolsky system
$$I(u)=\frac{1}{2}\int_{\mathbb{R}^{3}}|\nabla u|^{2}dx+\frac{1}{4}\int_{\mathbb{R}^{3}}
\phi_{u}u^{2}dx-\frac{1}{p}\int_{\mathbb{R}^{3}}|u|^{p}dx$$
constrained on $B_{\rho}=\left\{u\in H^{1}(\mathbb{R}^{3},C):\ \left\|u\right\|_{2}=\rho\right\},$ where $\rho>0.$ As such constrained problem $I(u)$ is bounded from below on $B_{\rho}$ when $p\in(2,\frac{10}{3}).$ We use minimizing method to get a normalized solution.
\end{abstract}

\noindent {\bf Key words:}Normalized solutions, Schr\"{o}dinger-Bopp-Podolsky system, Constrained minimization

\noindent {\bf 2010 Mathematics Subject Classification:}35Q55, 35J30, 35J50

\section{Introduction}
In this paper, we consider the following Schr\"{o}dinger-Bopp-Podolsky system:
\begin{eqnarray}\label{P1}
	\left\{   \begin{array}{ll}
	-\Delta u + \omega u +\phi u=|u|^{p-2}u   \ & \text{in} \  \mathbb{R}^{3}, \\
	-\Delta \phi + a^{2}\Delta^{2}\phi  = 4\pi u^{2}   \ & \text{in}\  \mathbb{R}^{3}.\\
\end{array}
\right.	
\end{eqnarray}
where $u,\phi:\mathbb{R}^{3}\to\mathbb{R}$, $\omega>0$ and let $a=1$. Such a system was first proposed by d'Avenia and Siciliano\cite{MR3957980} and can be used to describe solitary waves for nonlinear stationary equations of Schr\"{o}dinger type interacting with an electrostatic field in the Bopp-Podolsky electromagnetic theory. The Bopp-Podolsky theory, developed by Bopp\cite{MR0004193} and Podolsky\cite{MR6723} independently, can be interpreted as an effective theory for short distances and for large distance it is experimentally indistinguishable from Maxwell one. For more physical aspects of this system, please refer to \cite{born1933modified,MR26955,born1934foundations} and the references therein. Next we focus on the mathematical aspects of the problem.

In recent years, many people consider this system through using variational methods. Next we recall some previous results. At first, d'Avenia and Siciliano\cite{MR3957980} have been devoted to the following autonomous system with subcritical growth:
\begin{eqnarray}\label{P 2}
	\left\{   \begin{array}{ll}
	-\Delta u + \omega u +q^{2}\phi u=|u|^{p-2}u   \ & \text{in} \  \mathbb{R}^{3}, \\
	-\Delta \phi + a^{2}\Delta^{2}\phi  = 4\pi u^{2}   \ & \text{in}\  \mathbb{R}^{3}.\\
\end{array}
\right.	
\end{eqnarray}
where $\omega,a>0,q\neq 0$ and $p>2$. They obtain that problem \eqref{P 2} has a nontrivial solution when $p\in(3,6)$ and $q>0$ or $p\in(2,3]$ and $q>0$ small enough. In the radial case, they get the solutions tend to solutions of classical Schr\"{o}dinger-Poisson system as $a\to 0.$ Furthermore, they prove that problem \eqref{P 2} does not have a nontrivial solution when $p\ge 6$ through using a Poho\v{z}aev identity. Siciliano and Silva\cite{MR4119258}, by means of the fibering approach, prove the system \eqref{P 2} has no solutions at all for large enough of $q$ and has two radical solutions for small enough of $q$ when $p\in(2,3]$, which improve some results in \cite{MR3957980}. Furthermore, under the assumption that $\omega$ is replaced by a coercive potential $V(x)$ and $|u|^{p-2}u$ is replaced with $f(u)$ in \cite{MR4217993}, they proved that problem \eqref{P 2} has a ground state when $f(u)=|u|^{p-2}u+h(x)$ with $p\in[4,6)$ and at least one positive solution for $f(u)=P(x)u^{5}+\mu |u|^{p-2}u$ with $p\in(2,6).$ They also proved that the problem \eqref{P 2} possesses infinitely many nontrivial solutions when $f(u)$ satisfies the following conditions:
\begin{itemize}
  \item[$(f_1)$] $f(t)=-f(-t).$
  \item[$(f_2)$] There exist $t\in(1,5)$ such that $\lim_{t\to 0}\frac{f(t)}{t}=\lim_{t\to +\infty}\frac{f(t)}{|t|^{t}}=0.$
  \item[$(f_3)$] $\lim_{|t|\to +\infty}\frac{F(t)}{|t|^{4}}=\infty,$ and there exists $\mu\ge 4,\ \kappa>0$ such that $\mu F(t)\le tf(t)+\kappa t^{2},$ where $F(t)=\int_{0}^{t}f(r)dr.$
\end{itemize}
Afterwards, Chen and Tang\cite{MR4050783} study the following critical Schr\"{o}dinger-Bopp-Podolsky with subcritical perturbations of general function:
\begin{eqnarray}\label{P 3}
	\left\{   \begin{array}{ll}
	-\Delta u + V(x)u +\phi u=\mu f(u)+u^{5}   \ & \text{in} \  \mathbb{R}^{3}, \\
	-\Delta \phi + a^{2}\Delta^{2}\phi  = 4\pi u^{2}   \ & \text{in}\  \mathbb{R}^{3},\\
\end{array}
\right.	
\end{eqnarray}
get a nontrivial solution and ground state solution. Yang, Chen and Liu\cite{MR4143620} study the existence of nontrivial solution when $f(u)$ in \eqref{P 3} without any growth and Ambrosetti-Rabinowitz conditions. Later, Li, Pucci and Tang\cite{MR4129340}, by using the method of the Poho\v{z}aev-Nehari manifold, get a nontrivial ground state solution for \eqref{P 3} when $f(u)=|u|^{p-1}u$.

Observe all the above articles, it is not difficult to find that they all regard $\omega$ as a fixed frequency parameter to seek nontrivial solutions. Hence, nothing can be given a priori on the $L^{2}$-norm of the solutions. However, many physicists are very interested in normalized solutions, i.e. solutions with prescribed $L^{2}$-norm. To my best knowledge, for solving the normalized solution of the Schr\"{o}dinger-Bopp-Podolsky system, only paper \cite{afonso2021normalized} consider it under Neumann boundary conditions. In the present paper, we study whether \eqref{P1} has normalized solutions without Neumann boundary conditions. A normalized solution of \eqref{P1} can be obtained as a constrained critical point of the functional
\begin{align}\label{A13}
I(u)=\frac{1}{2}\int_{\mathbb{R}^{3}}|\nabla u|^{2}dx+\frac{1}{4}\int_{\mathbb{R}^{3}}
\phi_{u}u^{2}dx-\frac{1}{p}\int_{\mathbb{R}^{3}}|u|^{p}dx
\end{align}
on the $L^{2}-$sphere in $H^{1}(\mathbb{R}^{3},C)$
$$B_{\rho}=\left\{u\in H^{1}(\mathbb{R}^{3},C):\ \left\|u\right\|_{2}=\rho\right\},$$
where $\phi_{u}$ is defined in Section 2. It should be noted that in this case the frequency $\omega$ is no longer be imposed but instead appears as a Lagrange parameter.

Consider the following minimization problem
$$I_{\rho^{2}}=\inf_{B_{\rho}}I(u),$$
which makes sense for $p\in(2,\frac{10}{3}).$ Moreover, in this case, we know that the $C^{1}$ functional $I$ is bounded from below and coercive on $B_{\rho}$. The detailed proof procedure can be seen in Lemma \ref{L 1}. Hence, the most dramatic difficult problem is the minimizing sequences $\left\{u_{n}\right\}\subset B_{\rho}$ the lack of compactness. In fact, there will be two bad possibilities, namely
\begin{enumerate}[(i)]
		\item $u_{n}\rightharpoonup 0$.
		\item $u_{n}\rightharpoonup\hat{u}\neq 0$ and $0<\left\|\hat{u}\right\|_{2}<\rho.$
	\end{enumerate}
We can rule out both cases using the general method. i.e. first prove any minimizing sequence weakly converges, up to translation, to a function $\hat{u}$ which is different from zero, excluding the vanishing case and then we have to show $\left\|\hat{u}\right\|_{2}=\rho.$ In order to prove $\left\|\hat{u}\right\|_{2}=\rho,$ we need to know that function $I$ satisfies the strong subadditivity inequality, namely
\begin{align}\label{A12}
I_{\rho^{2}}<I_{\mu^{2}}+I_{\rho^{2}-\mu^{2}}\ \ \ \ \ for\ all\ 0<\mu<\rho.
\end{align}
When $p\in(3,\frac{10}{3})$, we adapt a similar argument as in \cite{MR2786153}. By standard scaling arguments and some basic properties of functions we can get \eqref{A12}. However, when $p\in(2,3)$, it will be difficult to prove \eqref{A12} using the above method. Therefore, we use the method of \cite{MR2826402}. We must show that

$(\mathbf{HD})$ the function $s\mapsto\frac{I_{s^{2}}}{s^{2}}$ is monotone decreasing.

In fact, if $(\mathbf{HD})$ hold when $\mu\in(0,\rho),$ from direct calculation we obtain
$$\frac{\mu^{2}}{\rho^{2}}I_{\rho^{2}}<I_{\mu^{2}}\ \ \ \ \ \frac{\rho^{2}-\mu^{2}}{\rho^{2}}<I_{\rho^{2}-\mu^{2}}.$$
Therefore,
$$I_{\rho^{2}}=\frac{\mu^{2}}{\rho^{2}}I_{\rho^{2}}+\frac{\rho^{2}-\mu^{2}}{\rho^{2}}
I_{\rho^{2}}<I_{\mu^{2}}+I_{\rho^{2}-\mu^{2}}.$$
Proving that satisfying $(\mathbf{HD})$ is not an easy task. The function $s\mapsto\frac{I_{s^{2}}}{s^{}2}$ have a fast oscillating behavior, even in a neighborhood of the origin. We use a large class of functionals provided that they satisfy some good scaling properties and give sufficient condition to guarantee $(\mathbf{HD})$ hold.

Our results are as follows.
\begin{theorem}\label{T1}
	 Let $p\in (3,\frac{10}{3})$ and $\rho>0$, there exists $\rho_{1}>0$ such that for any $\rho\in(\rho_{1},+\infty)$ there exists a couple $(u_{\rho},\lambda_{\rho})\in H^{1}(\mathbb{R}^{3})\times\mathbb{R}^{+}$ solution of \eqref{P1} with $\left\|u_{\rho}\right\|_{2}=\rho.$
\end{theorem}

\begin{theorem}\label{T2}
	 Let $p\in (2,3)$ and $\rho>0$, there exists $\rho_{2}>0$ such that for any $\rho\in(0,\rho_{2})$ there exists a couple $(u_{\rho},\lambda_{\rho})\in H^{1}(\mathbb{R}^{3})\times\mathbb{R}^{+}$ solution of \eqref{P1} with $\left\|u_{\rho}\right\|_{2}=\rho.$
\end{theorem}
Finally, we prove the orbital stability of standing waves for Schr\"{o}dinger-Bopp-Podolsky system. It is well-known that the original approach of proving orbital stability is very complicated. Therefore, we draw on the method of \cite{MR677997} to get the following theorem.
\begin{theorem}\label{T3}
Let $p\in(2,\frac{10}{3}).$ Then the set
$$S_{\rho}=\left\{e^{i\theta}u(x):\theta\in[0,2\pi),\left\|u\right\|_{2}=\rho,I(u)=I_{\rho^{2}}\right\}$$
is orbitally stable.
\end{theorem}
\begin{remark}
Notice that if $p\ge\frac{10}{3},$ by observing Lemma \ref{L 1}, we know the functional $I$ is unbounded from below on $B_{\rho}.$ Therefore, the minimizing method will no longer apply. At this point we know that if $a\to 0$ in problem \eqref{P1} the Schr\"{o}dinger-Bopp-Podolsky system becomes the classical Schr\"{o}dinger-Poisson system, namely
\begin{eqnarray}\label{P5}
	\left\{   \begin{array}{ll}
	-\Delta u + \omega u +\phi u=|u|^{p-2}u   \ & \text{in} \  \mathbb{R}^{3}, \\
	-\Delta \phi = 4\pi u^{2}   \ & \text{in}\  \mathbb{R}^{3}.\\
\end{array}
\right.	
\end{eqnarray}
Regarding the normalized solution of the mass sup-criticality of problem \eqref{P5}, the paper \cite{MR3092340} investigates the case $p\in(\frac{10}{3},6).$ They proved that problem \eqref{P5} has a mountain-pass geometry and similar to \cite{MR695536} develop a deformation argument to get a localization of Palais--Smale sequence. However, the mountain-pass geometry can not guarantee the existence of a bounded Palais--Smale sequence, so a set of constraints about Poho\v{z}aev equality was constructed to solve this problem. Moreover, concerning the lack of compactness for Palais--Smale sequence, it does not to reduce the problem to the classical vanishing-dichotomy-compactness scenario or like this paper method to check of the strong subadditivity inequalities. They studied some properties about the mountain-pass level function to overcome this difficulty. Afterwards, using the above geometric structure and adding a new constraint $\int_{\mathbb{R}^{3}}|\nabla u|^{2}dx<\frac{3}{5}\int_{\mathbb{R}^{3}}|u|^{\frac{10}{3}}dx,$ Ye \cite{MR3658970} studied the existence and the concentration behavior for problem \eqref{P5} when $p=\frac{10}{3}.$ In addition, those papers \cite{MR3803765,MR4008541,MR4085499,MR4318833} has also studied the normalized solution of problem \eqref{P5} under different suitable assumptions. Unfortunately, None of the methods in the above papers can be used in the Schr\"{o}dinger-Bopp-Podolsky system. Because if we do the scaling $u_{\lambda}(\cdot)=\lambda^{\alpha}u(\lambda^{\beta}(\cdot)),$ $\alpha,\beta\in\mathbb{R},$ $\lambda>0$ of the nonlocal term $\phi_{u}$, form \eqref{A16} we know it is impossible to get $\phi_{u_{\lambda(\cdot)}}=\lambda^{\xi}\phi_{u}$ where $\xi=\xi(\alpha,\beta).$ Moreover, $\int_{\mathbb{R}^{3}}\int_{\mathbb{R}^{3}}
e^{-|x-y|}u^{2}(x)u^{2}(y)dxdy$ in Poho\v{z}aev equality is also difficult to handle.
Therefore, we can not study the normalized solution for problem \eqref{P1} when $p\ge\frac{10}{3}.$
\end{remark}

\section{Preliminaries}
In this section we explain the notations and some auxiliary lemmas which are useful later. We cite the book \cite{MR3890060} for the standard reference book.

Let $H^{1}(\mathbb{R}^{3})$ denote the usual Sobolev space with respect to the norm $$\left\|u\right\|_{H^{1}(\mathbb{R}^{3})}^{2}=\int_{\mathbb{R}^{3}}(\left |\nabla u \right |^{2}+u^{2})dx.$$
$L^{p}(\mathbb{R}^{3})$ denote the usual Lebesgue space with respect to the norm $$\left\|u\right\|_{p}^{p}=\int_{\mathbb{R}^{3}}|u|^{p}dx.$$
$D^{1,2}(\mathbb{R}^{3})$ is the completion of $C_{0}^{\infty}$ with respect with
$$\left\|u\right\|_{D^{1,2}(\mathbb{R}^{3})}^{2}=\int_{\mathbb{R}^{3}}|\nabla u|^{2}dx.$$
The $\mathcal{D}$ is defined by the completion of $C_{0}^{\infty}(\mathbb{R}^{3})$ equipped with the scalar product
$$\langle u,v\rangle_{\mathcal{D}}:=\int_{\mathbb{R}^{3}}\nabla u\nabla vdx+\int_{\mathbb{R}^{3}}\Delta u\Delta vdx.$$
It is easy to know that $\mathcal{D}(\mathbb{R}^{3})$ is continuously embedded in $D^{1,2}(\mathbb{R}^{3})$ and consequently in $L^{\infty}(\mathbb{R}^{3}).$

$C_{1},\ C_{2},\ \cdot\cdot\cdot$ denote positive constant possibly different in different places.

For the sake of brevity, from now on we define the following quantities:
\begin{align*}
\begin{split}
&A(u):=\frac{1}{2}\int_{\mathbb{R}^{3}}|\nabla u|^{2}dx\ \ \ \ \ \ \ \ B(u):=\frac{1}{4}\int_{\mathbb{R}^{3}}\phi_{u}|u|^{2}dx\\
&C(u):=-\frac{1}{p}\int_{\mathbb{R}^{3}}|u|^{p}dx\ \ \ \ \ \ \ \ T(u):=B(u)+C(u)
\end{split}
\end{align*}
so that
$$I(u)=A(u)+T(u).$$
\begin{definition}
Let $u\in H^{1}(\mathbb{R}^{3}),$ $u\neq 0.$ A continuous path $g_{u}:\theta\in\mathbb{R}^{+}\mapsto g_{u}(\theta)\in H^{1}(\mathbb{R}^{3})$ such that $g_{u}(1)=u$ is said to be a scaling path of $u$ if
$\Theta_{g_{u}}(\theta):=\frac{\left\|g_{u}(\theta)\right\|_{2}^{2}}{\left\|u\right\|_{2}^{2}}$ is differentiable and $\Theta'_{g_{u}}(1)\neq 0.$ We denote with $G_{u}$ the set of the scaling paths of $u$.
\end{definition}
Moreover, we define a real valued function when $u\neq 0$, which will be useful in subsequent proofs
\begin{align*}
h_{g_{u}}(\theta):=I(g_{u}(\theta))-\Theta_{g_{u}}(\theta)I(u),\ \ \ \theta\ge 0.
\end{align*}
\begin{definition}
Let $u\neq 0$ be fixed and $g_{u}\in G_{u}.$ We say that the scaling path $g_{u}$ is admissible for the functional $I$ if $h_{g_{u}}$ is a differentiable function.
\end{definition}
\begin{proposition}(\cite{MR2786153})\label{P4}
Let $T$ be a $C^{1}$ functional on $H^{1}(\mathbb{R}^{3})$ and $\left\{u_{n}\right\}\subset B_{\rho}$ be a minimizing sequence for $I_{\rho}$ such that $u_{n}\rightharpoonup u\neq 0$; Assume
\begin{enumerate}[(i)]
    \item $\langle T'(u_{n}),u_{n}\rangle=O(1)$;
	\item $T(u_{n}-u)+T(u)=T(u_{n})+o(1)$;
    \item $T(\alpha_{n}(u_{n}-u))-T(u_{n}-u)=o(1),$ where $\alpha_{n}=\frac{\rho^{2}-\mu^{2}}{\left\|u_{n}-u\right\|_{2}}$;
    \item $\langle T'(u_{n})-T'(u_{m}),u_{n}-u_{m}\rangle=o(1);$
    \item $I_{\rho^{2}}< I_{\mu^{2}}+I_{\rho^{2}-\mu^{2}}$ for any $0<\mu<\rho.$
\end{enumerate}
Then $\left\|u_{n}-u\right\|_{H^{1}(\mathbb{R}^{3})}\to 0.$
\end{proposition}
\begin{lemma}(\cite[Lemma 3.2]{MR3957980})
The space $C_{0}^{\infty}(\mathbb{R}^{3})$ is dense in $\mathcal{A}:=\left\{\phi\in D^{1,2}(\mathbb{R}^{3}):\Delta\phi\in L^{2}(\mathbb{R}^{3})\right\}$ normed by $\sqrt{\langle\phi,\phi\rangle_{\mathcal{D}}}$. Hence, $\mathcal{D}=\mathcal{A}.$
\end{lemma}
From \cite{MR3957980} we obtain that there exists a unique solution $\phi_{u}\in\mathcal{D}$ to the second equation in \eqref{P1}. Its expression is
\begin{align}
\phi_{u}(x):=\int_{\mathbb{R}^{3}}\frac{1-e^{-|x-y|}}{|x-y|}u^{2}(y)dy.
\end{align}
Notice that if we assume $u_{\lambda}(\cdot)=\lambda^{\alpha}u(\lambda^{\beta}(\cdot)),$ $\alpha,\beta\in\mathbb{R},$ $\lambda>0,$ then
\begin{align}\label{A16}
\begin{split}
\phi_{u_{\lambda}}(x)&=\int_{\mathbb{R}^{3}}\lambda^{2\alpha+\beta}
(1-e^{-\frac{1}{\lambda^{\beta}}|\lambda^{\beta}x-\lambda^{\beta}y|})
\frac{|u(\lambda^{\beta}y)|^{2}}{|\lambda^{\beta}x-\lambda^{\beta}y|}dy\\
&=\lambda^{2(\alpha-\beta)}\int_{\mathbb{R}^{3}}
(1-e^{-\frac{1}{\lambda^{\beta}}|\lambda^{\beta}x-y|})
\frac{|u(y)|^{2}}{|\lambda^{\beta}x-y|}dy.
\end{split}
\end{align}
Moreover $\phi_{u}$ has the following properties.
\begin{lemma}(\cite{MR3957980}\label{L2})
For every $u\in H^{1}(\mathbb{R}^{3}),$ we have:
\begin{enumerate}[(i)]
		\item $\phi_{u}\ge 0$ for all $u\in H^{1}(\mathbb{R}^{3})$.
		\item $\left\|\phi_{u}\right\|_{\mathcal{D}}\le C\left\|u\right\|^{2}$ and
$\int_{\mathbb{R}^{3}}\phi_{u}u^{2}dx\le C\left\|u\right\|_{\frac{12}{5}}^{4}.$
		\item if $u_{n}\rightharpoonup u$ in $H^{1}(\mathbb{R}^{3}),$ then $\phi_{u_{n}}\rightharpoonup\phi_{u}$ in $\mathcal{D}.$
	\end{enumerate}
\end{lemma}
\begin{lemma}\label{P2}
Assume $I$ satisfied the following conditions
\begin{enumerate}[(i)]
    \item $I$ satisfied the weak subadditivity inequality, namely:
        \begin{align}\label{A14}
        I_{\rho^{2}}\le I_{\mu^{2}}+I_{\rho^{2}-\mu^{2}}\ \ \ \ \ for\ all\ 0<\mu<\rho.
        \end{align}
    \item $-\infty<I_{s^{2}}<0$ for all $s>0$.
	\item $s\longmapsto I_{s^{2}}$ is continuous.
    \item $\lim_{s\to 0}\frac{I_{s^{2}}}{s^{2}}=0$.
    \item $\forall u\in M(\rho):=\bigcup_{\mu\in(0,\rho]}
    \left\{u\in B_{\mu}:I(u)=I_{\mu^{2}}\right\}$, $\exists g_{u}\in G(u)$ admissible, such that $\frac{d}{d\theta}h_{g_{u}}(\theta)\mid_{\theta=1}\neq 0.$
\end{enumerate}
Then $(\mathbf{HD})$ holds. Hence, we know $I_{\rho^{2}}< I_{\mu^{2}}+I_{\rho^{2}-\mu^{2}}$ for any $0<\mu<\rho.$
\end{lemma}
\begin{proof}
The detailed proof process can be found in Theorem 2.1 of references \cite{MR2826402}.
\end{proof}
\begin{lemma}\label{L 1}
	If $p\in(2,\frac{10}{3}),$ then for every $\rho>0$ the functional $I$ is bounded from below and coercive on $B_{\rho}.$
\end{lemma}
\begin{proof}
Due to Gagliardo--Nirenberg inequality, we can get
\begin{align*}
\left\|u\right\|_{p}^{p}\le C_{p}\left\|u\right\|_{2}^{\frac{6-p}{2}}\left\|\nabla u\right\|_{2}^{\frac{3p}{2}-3}.
\end{align*}
From \eqref{A13} it is easily to know
\begin{align*}
    \begin{split}
    I(u)&\ge \frac{1}{2}\int_{\mathbb{R}^{3}}|\nabla u|^{2}dx-\frac{1}{p}\int_{\mathbb{R}^{3}}|u|^{p}dx\\&
    \ge \frac{1}{2}\left\|\nabla u\right\|_{2}^{2}-\frac{1}{p}C_{p}\left\|u\right\|_{2}^{\frac{6-p}{2}}
    \left\|\nabla u\right\|_{2}^{\frac{3p}{2}-3}.
    \end{split}
\end{align*}
Because of $p<\frac{10}{3},$ hence $\frac{3p}{2}-3<2,$ we can obtain
$$I(u)\ge\frac{1}{2}\left\|\nabla u\right\|_{2}^{2}+O(\left\|\nabla u\right\|_{2}^{2}).$$
\end{proof}
\begin{lemma}\label{L3}
If $p\in(2,\frac{10}{3})$, let ${u_{n}}\subset B_{\rho}$ be a minimizing sequence for $I_{\rho^{2}}$ such that $u_{n}\rightharpoonup u\neq 0,$ then $T$ satisfies the following properties
\begin{enumerate}[(i)]
    \item $\langle T'(u_{n}),u_{n}\rangle=O(1)$;
	\item $T(u_{n}-u)+T(u)=T(u_{n})+o(1)$;
    \item $T(\alpha_{n}(u_{n}-u))-T(u_{n}-u)=o(1),$ where $\alpha_{n}=\frac{\rho^{2}-\mu^{2}}{\left\|u_{n}-u\right\|_{2}}$;
    \item $\langle T'(u_{n})-T'(u_{m}),u_{n}-u_{m}\rangle=o(1)$.
\end{enumerate}
\end{lemma}
\begin{proof}
Due to $u_{n}$ be a minimizing sequence, we can easily to deduce that $\langle T'(u_{n}),u_{n}\rangle=O(1).$ From Lemma \ref{L 1} we know that ${u_{n}}$ is bounded in $H^{1}(\mathbb{R}^{3})$. Therefore, ${u_{n}}$ is bounded in $L^{s}$ for $s\in[2,2^{*}]$ and there exists $u\in H^{1}(\mathbb{R}^{3})$ such that $u_{n}\rightharpoonup u$ in $H^{1}(\mathbb{R}^{3})$ and $u_{n}\to u$ a.e. in $\mathbb{R}^{3}.$ We give these notations
$$G(x,y)=\frac{1-e^{-|x-y|}}{|x-y|},\ \ \ \ \ \ \  A:=\int_{\mathbb{R}^{3}}\int_{\mathbb{R}^{3}}u^{2}(y)u^{2}(x)G(x,y)dxdy,$$
\begin{align*}
    \begin{split}
    &I_{n}^{(1)}:=\int_{\mathbb{R}^{3}}\int_{\mathbb{R}^{3}}u^{2}_{n}(y)u^{2}(x)G(x,y)dxdy,\\
    &I_{n}^{(2)}:=\int_{\mathbb{R}^{3}}\int_{\mathbb{R}^{3}}u_{n}(y)u(y)u_{n}(x)u(x)G(x,y)dxdy,\\
    &I_{n}^{(3)}:=\int_{\mathbb{R}^{3}}\int_{\mathbb{R}^{3}}u_{n}^{2}(y)u_{n}(x)u(x)G(x,y)dxdy,\\
    &I_{n}^{(3)}:=\int_{\mathbb{R}^{3}}\int_{\mathbb{R}^{3}}u_{n}(y)u(y)u^{2}(x)G(x,y)dxdy.
    \end{split}
\end{align*}
From direct calculation we can deduce
\begin{align*}
B(u_{n}-u)-(B(u_{n})-B(u))=2I_{n}^{(1)}+4I_{n}^{(2)}-4I_{n}^{(3)}-4I_{n}^{(4)}+2A.
\end{align*}
Next we will show that $\lim_{n\to\infty}I_{n}^{(i)}=A,\ \ i=1,2,3,4.$ Let
$$v_{n}(x):=\int_{\mathbb{R}^{3}}\frac{1-e^{-|x-y|}}{|x-y|}u^{2}_{n}(y)dy,\ \ \ v(x)=\int_{\mathbb{R}^{3}}\frac{1-e^{-|x-y|}}{|x-y|}u^{2}(y)dy.$$
\begin{align*}
\begin{split}
|v_{n}(x)-v(x)|&\le\int_{\mathbb{R}^{3}}|u_{n}^{2}(y)-u^{2}(y)|\frac{1-e^{-|x-y|}}{|x-y|}dy\\
&\le\int_{\mathbb{R}^{3}}|u_{n}^{2}(y)-u^{2}(y)|\frac{1}{|x-y|}dy\\
&\le C\left\|u_{n}^{2}-u^{2}\right\|_{L^{2}(B_{R}(x))}\left(\int_{|y-x|\le R}\frac{1}{|x-y|^{2}}dy\right)^{\frac{1}{2}}\\&\ \ \ +C\left\|u_{n}^{2}-u^{2}\right\|_{L^{\frac{4}{3}}(B^{c}_{R}(x))}\left(\int_{|y-x|\ge R}\frac{1}{|x-y|^{2}}dy\right)^{\frac{1}{2}}.
\end{split}
\end{align*}
Letting $n\to\infty$ and $R\to\infty$, we know $v_{n}\to v$ a.e. on $\mathbb{R}^{3}.$
Due to $v_{n}=\phi_{u_{n}}\in\mathcal{D}(\mathbb{R}^{3})$ and Lemma \ref{L2}, we know
\begin{align*}
\left\|v_{n}\right\|_{6}=C \left\|v_{n}\right\|_{D}\le C\left\|u_{n}\right\|_{\frac{12}{5}}^{2}\le C.
\end{align*}
Hence we can assume $v_{n}\rightharpoonup v$ in $L^{6}(\mathbb{R}^{3})$ and due to $u\in H^{1}(\mathbb{R}^{3})$ then $u^{2}\in L^{\frac{6}{5}}(\mathbb{R}^{3}),$ we can get
\begin{align*}
\int_{\mathbb{R}^{3}}v_{n}u^{2}dx\to\int_{\mathbb{R}^{3}}vu^{2}dx.
\end{align*}
We prove that $\lim_{n\to\infty}I_{n}^{(1)}=A.$ Next we prove $\lim_{n\to\infty}I_{n}^{(2)}=A.$ we set
\begin{align*}
\hat{v}_{n}(x):=\int_{\mathbb{R}^{3}}\frac{1-e^{-|x-y|}}{|x-y|}u_{n}(y)u(y)dy.
\end{align*}
A similar proof process to the previous one can easily obtain $\hat{v}_{n}\to v$ a.e. in $\mathbb{R}^{3}.$
From Hardy-Littlewood-Sobolev inequality, $\hat{v}_{n}\in L^{6}(\mathbb{R}^{3})$ and
\begin{align*}
\left\|\hat{v}_{n}\right\|_{6}\le\int_{\mathbb{R}^{3}}\frac{u_{n}(y)u(y)}{|x-y|}dy\le C\left\|u_{n} u\right\|_{\frac{6}{5}}\le C\left\|u_{n}\right\|_{\frac{12}{5}}\left\|u\right\|_{\frac{12}{5}}.
\end{align*}
So we get
\begin{align*}
\left\|\hat{v}_{n}u_{n}\right\|_{2}\le\left\|\hat{v}_{n}\right\|_{6}\left\|u_{n}\right\|_{3}
\le C\left\|u_{n}\right\|^{2}\left\|u\right\|\le C.
\end{align*}
Therefore up to a subsequence $\hat{v}_{n}u_{n}\rightharpoonup vu$ in $L^{2}(\mathbb{R}^{3}).$ Due to $u\in L^{2}(\mathbb{R}^{3})$ we get
\begin{align*}
\int_{\mathbb{R}^{3}}\hat{v}_{n}u_{n}udx\to\int_{\mathbb{R}^{3}}vu^{2}dx.
\end{align*}
Therefore, we prove that $\lim_{n\to\infty}I_{n}^{(2)}=A.$
A similar proof process can be deduced $\lim_{n\to\infty}I_{n}^{(3)}=A$ and $\lim_{n\to\infty}I_{n}^{(4)}=A.$
Hence we get
\begin{align*}
B(u_{n}-u)+B(u)=B(u_{n})+o(1).
\end{align*}
From Brezis-Leib lemma can deduce
\begin{align*}
C(u_{n}-u)+C(u)=C(u_{n})+o(1).
\end{align*}
$(ii)$ certified.

From Sobolev inequality and Gagliardo--Nirenberg inequality, we know
\begin{align*}
B(u_{n})=\int_{\mathbb{R}^{3}}\phi_{u_{n}}u_{n}^{2}dx\le C\left\|u_{n}\right\|_{\frac{12}{5}}^{4}\le C\left\|u_{n}\right\|_{2}^{3}\left\|\nabla u_{n}\right\|_{2}^{3}.
\end{align*}
Due to Brezis-Leib lemma, we know
\begin{align*}
\left\|u_{n}-u\right\|_{2}^{2}+\left\|u\right\|_{2}^{2}=\left\|u_{n}\right\|_{2}^{2}+o(1),
\end{align*}
hence $$\alpha_{n}=\frac{\rho^{2}-\mu^{2}}{\left\|u_{n}-u\right\|_{2}}\to 1.$$
We can know
\begin{align*}
\begin{split}
&B(\alpha_{n}(u_{n}-u))-B(u_{n}-u)=(\alpha_{n}^{4}-1)B(u_{n}-u)=o(1)\\
&C(\alpha_{n}(u_{n}-u))-C(u_{n}-u)=(\alpha_{n}^{p}-1)C(u_{n}-u)=o(1).
\end{split}
\end{align*}
$(iii)$ certified.
Due to $u_{n},u_{m}$ are minimizing sequence and Gagliardo--Nirenberg inequality we know
\begin{align*}
\left\|u_{n}-u_{m}\right\|_{p}\le\left\|u_{n}-u_{m}\right\|_{2}^{\frac{6-p}{2p}}
\left\|u_{n}-u_{m}\right\|_{2}^{\frac{3}{2}-3p}=o(1).
\end{align*}
From H\"{o}lder inequality we deduce
$$\int_{\mathbb{R}^{3}}|u_{n}|^{p-1}|u_{n}-u|dx\le
\left(\int_{\mathbb{R}^{3}}|u_{n}|^{p}dx\right)^{\frac{1}{q}}
\left(\int_{\mathbb{R}^{3}}|u_{n}-u|^{p}dx\right)^{\frac{1}{p}}=o(1),$$
where $q=\frac{p}{p-1}.$ Hence, we know
$$\left|\int_{\mathbb{R}^{3}}\left(|u_{n}|^{p-1}-|u_{m}|^{p-1}\right)
\left(u_{n}-u_{m}\right)dx\right|\le
C\left\|u_{n}-u_{m}\right\|_{p}=o(1).$$
We can obtain $\langle C'(u_{n})-C'(u_{m}),u_{n}-u_{m}\rangle=o(1),$ then we need to verification $B.$
From H\"{o}lder inequality we deduce
\begin{align}
\begin{split}
\int_{\mathbb{R}^{3}}\phi_{u_{n}}u_{n}(u_{n}-u_{m})dx&\le C\left\|\phi_{u_{n}}\right\|_{6}\left\|u_{n}\right\|_{2}\left\|
u_{n}-u_{m}\right\|_{3}\\
&\le C \left\|u_{n}\right\|^{2}\left\|u_{n}\right\|_{2}\left\|u_{n}-u_{m}\right\|_{3}\\&=o(1).
\end{split}
\end{align}
Hence, we get
$$\langle T'(u_{n})-T'(u_{m}),u_{n}-u_{m}\rangle=o(1).$$
\end{proof}
\begin{lemma}\label{L4}
For every $\rho>0,$ $\left\{u_{n}\right\}$ be a minimizing sequence in $B_{\rho}$ with $I_{\rho^{2}}<0$, then $u_{n}\rightharpoonup u\neq 0.$
\end{lemma}
\begin{proof}
Due to $\left\{u_{n}\right\}$ be a minimizing sequence in $B_{\rho}$ for $I_{\rho^{2}}$ and translation does not deform, we know for any sequence $\left\{y_{n}\right\}\subset\mathbb{R}^{3},$ $u(\cdot+y_{n})$ is still a minimizing sequence for $I_{\rho^{2}}$. From Lions' lemma, if
$$\lim_{n\to\infty}\left(\sup_{y\in\mathbb{R}^{3}}\int_{B(y,1)}|u_{n}|^{2}dx\right)=0,$$
where $B(a,r)=\left\{x\in\mathbb{R}^{3},\ |x-a|\le r\right\}.$ Then $u_{n}\to 0$ in $L^{q}(\mathbb{R}^{3})$ for $q\in(2,2^{*}),$ therefore $C(u_{n})\to 0.$  Moreover, From $I_{\rho^{2}}<0$ we obtained
$$\sup_{y\in\mathbb{R}^{3}}\int_{B(y,1)}|u_{n}|^{2}dx\ge\mu>0,$$
hence we can choose $y_{n}\in\mathbb{R}^{3}$ such that
$$\int_{B(0,1)}|u_{n}(\cdot+y_{n})|^{2}dx\ge\mu>0.$$
Since $H^{1}(B(0,1))\hookrightarrow L^{2}(B(0,1))$ is a compact embedding, the weak limit of $u_{n}(\cdot+y_{n})$ is nonzero. So it can be deduced that $u_{n}\rightharpoonup u\neq 0.$
\end{proof}
\section{Proof of the Theorem 1.1}

\begin{proof}[\textbf{Proof of Theorem 1.1}]
Now, proof Theorems \ref{T1} only need to prove that function $I$ satisfies the strong subadditivity inequality, namely:

There exists $\rho_{1}>0$ such that $I_{\mu^{2}}<0$ for all $\mu\in(\rho_{1},+\infty)$ and
$$I_{\rho^{2}}<I_{\mu^{2}}+I_{\rho^{2}-\mu^{2}},$$
for all $\rho>\rho_{1}$ and $0<\mu<\rho.$
We define $u_{\theta}(x)=\theta^{1-\frac{3}{2}\beta}u(\frac{x}{\theta^{\beta}}),$ then from direct calculations we can get
\begin{align*}
\begin{split}
&\left\|u_{\theta}\right\|_{2}=\theta\left\|u\right\|_{2},\\
&A(u_{\theta})=\frac{1}{2}\int_{\mathbb{R}^{3}}|\nabla u_{\theta}|^{2}dx=\theta^{2-2\beta}A(u),\\
&B(u_{\theta})=\int_{\mathbb{R}^{3}}\int_{\mathbb{R}^{3}}
\frac{1-e^{-|x-y|}}{|x-y|}u_{\theta}^{2}(x)u_{\theta}^{2}(y)dxdy,\\&
\ \ \ \ \ \ \ \ \le
\int_{\mathbb{R}^{3}}\int_{\mathbb{R}^{3}}
\frac{1}{|x-y|}u_{\theta}^{2}(x)u_{\theta}^{2}(y)dxdy=\theta^{4-\beta}H(u),\\
&C(u_{\theta})=\theta^{(1-\frac{3}{2}\beta)p+3\beta}C(u),
\end{split}
\end{align*}
where $H(u)=\int_{\mathbb{R}^{3}}
\int_{\mathbb{R}^{3}}\frac{1}{|x-y|}u^{2}(x)u^{2}(y)dxdy.$
We can also know
\begin{align*}
\begin{split}
I(u_{\theta})&=A(u_{\theta})+B(u_{\theta})+
C(u_{\theta})\\
&\le\theta^{2-2\beta}A(u)+\theta^{4-\beta}H(u)+
\theta^{(1-\frac{3}{2}\beta)p+3\beta}C(u)\\
&=\theta^{2}\left(I(u)+(\theta^{-2\beta}-1)A(u)+
(\theta^{(1-\frac{3}{2}\beta)p+3\beta-2}-1)
C(u)+\theta^{2-\beta}H(u)-B(u)\right)\\
&=\theta^{2}(I(u)+f(\theta,u)),
\end{split}
\end{align*}
where $f(\theta,u)=(\theta^{-2\beta}-1)A(u)+(\theta^{(1-\frac{3}{2}\beta)p+3\beta-2}-1)
C(u)+\theta^{2-\beta}H(u)-B(u).$

When $\beta=-2$, we get
$$I(u_{\theta})\le\theta^{6}A(u)+\theta^{6}H(u)+\theta^{4p-6}C(u),$$
and $4p-6>6$ for $3<p<\frac{10}{3}.$ Hence, for $\theta$ sufficiently large we have $I(u_{\theta})\to 0^{-}.$
Let $\left\{u_{n}\right\}$ be a minimizing sequence in $B_{\rho}$ with $I_{\rho^{2}}<0,$ then
\begin{align*}
\begin{split}
&0<k_{1}<A(u_{n})<k_{2},\\
&0<\eta_{1}<|C(u_{n})|<\eta_{2}.
\end{split}
\end{align*}
Indeed, if $A(u_{n})=o(1)$ we have $|C(u_{n})|=o(1)$ and $I_{\rho^{2}}=0.$ For $\beta=-2$ we get
$$f(\theta,u)=(\theta^{4}-1)A(u)+(\theta^{4p-8}-1)C(u)+\theta^{4}H(u)-B(u),$$
with $4p-8>4$ and
\begin{align*}
\begin{split}
&f'(\theta,u)\mid_{\theta=1}=4A(u)+4H(u)+(4p-8)C(u)<k<0,\\
&f''(\theta,u)=12\theta^{2}A(u)+(4p-8)(4p-9)\theta^{4p-10}C(u)+12\theta^{2}H(u)<k<0.
\end{split}
\end{align*}
we know $f(1,u)=H(u)-B(u)>0$ and $f(\theta,u_{n})<k(\theta)<0$ for any $\theta>1$. From the mean value theorem, we can get there exists $\theta_{0}$ such that if $\theta>\theta_{0}$ then $f(\theta,u)<0.$ hence
$$I_{\theta^{2}\rho^{2}}<\theta^{2}I(u_{n})=\theta^{2}I_{\rho^{2}}.$$
Let us suppose that $\mu^{2}<\rho^{2}-\mu^{2}$. We distinguish three cases
\begin{align*}
\begin{split}
&\bullet\ \  \mu^{2}<\rho^{2}-\mu^{2}<\rho_{1}^{2}\\
&\bullet\ \  \mu^{2}<\rho_{1}^{2}<\rho^{2}-\mu^{2}\\
&\bullet\ \  \rho_{1}^{2}<\mu^{2}<\rho^{2}-\mu^{2}.
\end{split}
\end{align*}
The first case is trivial. For the second one, we know $I_{\rho^{2}-\mu^{2}}>I_{\rho^{2}}$ and we conclude. For the third case
\begin{align}
\begin{split}
I_{\rho^{2}}&=I_{\frac{\rho^{2}}{\mu^{2}}\mu^{2}}<\frac{\rho^{2}}{\mu^{2}}I_{\mu^{2}}\\
&=\frac{\rho^{2}-\mu^{2}+\mu^{2}}{\mu^{2}}I_{\mu^{2}}\\
&=\frac{\rho^{2}-\mu^{2}}{\mu^{2}}
I_{\frac{\mu^{2}}{\rho^{2}-\mu^{2}}\rho^{2}-\mu^{2}}
+I_{\mu^{2}}\le I_{\mu^{2}}+I_{\rho^{2}-\mu^{2}}.
\end{split}
\end{align}
Since the strong subadditivity inequality condition holds, then combine Lemma \ref{L3} with Lemma \ref{L4}, we can apply Proposition \ref{P4} and conclude the proof of Theorem \ref{T1}.
\end{proof}
\section{Proof of the Theorem 1.2}
When $p\in(2,3)$ we apply Lemma \ref{P2} to proof strong subadditivity inequality. Hence the next step will be divided into several lemmas to verify all the assumptions in Lemma \ref{P2} in turn.
\begin{lemma}
$-\infty<I_{s^{2}}<0$ for all $s>0$ and condition \eqref{A14} holds.
\end{lemma}
\begin{proof}
The detailed proof process for the establishment of condition \eqref{A14} can be proved using the same way in Proposition 2.3 in \cite{MR2032129}. From Lemma \eqref{L 1} we can get $I_{s^{2}}>-\infty.$ Hence, we just need to proof $I_{s^{2}}<0.$
Let $g_{u}(\theta)=\theta^{1-\frac{3}{2}\beta}u(\frac{x}{\theta^{\beta}}).$ From calculation, we obtained $\Theta_{g_{u}}(\theta)=\theta^{2}$ and $\left\|g_{u}(\theta)\right\|_{2}=\theta.$ By variable substitution
\begin{align*}
\begin{split}
&A(g_{u}(\theta))=\theta^{2-2\beta}A(u),\\
&C(g_{u}(\theta))=\theta^{(1-\frac{3}{2}\beta)p+3\beta}C(u),\\
&B(g_{u}(\theta))=\theta^{4-\beta}\int_{\mathbb{R}^{3}}\int_{\mathbb{R}^{3}}
\frac{1-e^{-\theta^{\beta}}|x-y|}{|x-y|}u(x)^{2}u(y)^{2}dxdy.
\end{split}
\end{align*}
Assume $\beta=-2$ we have
\begin{align*}
I(g_{u}(\theta))=\frac{\theta^{6}}{2}A(u)+\frac{\theta^{6}}{4}
\int_{\mathbb{R}^{3}}\int_{\mathbb{R}^{3}}\frac{1-e^{-\theta^{-2}|x-y|}}{|x-y|}
u(x)^{2}u(y)^{2}dxdy+\frac{\theta^{4p-6}}{p}C(u).
\end{align*}
Due to $p\in(2,3),$ we know $4p-6<6$. Hence, $I(u_{\theta})\to 0^{-}$ as $\theta\to 0.$ Then there exists a small $\theta_{0}$ such that
$$I_{s^{2}}<0\ \ \ \forall s\in(0,\theta_{0}].$$
Let $\rho\in(\theta_{0},\sqrt{2}\theta_{0}),$ for every $s\in(\theta_{0},\rho)$ from adding conditions to get
$$I_{s^{2}}\le I_{\theta_{0}^{2}}+I_{s^{2}-\theta_{0}^{2}}<0.$$
Due to $s^{2}-\theta_{0}^{2}<\theta_{0}^{2}$ we know $I_{s^{2}}<0$ for $s$ in the large interval $(0,\rho].$ Repeat the above process to get $I_{s^{2}}<0$ for every $s>0.$
\end{proof}
\begin{lemma}
The function $s\mapsto I_{s^{2}}$ is continuous.
\end{lemma}
\begin{proof}
To prove this lemma just prove if $\rho_{n}\to\rho$ then $\lim_{n\to\infty}I_{\rho_{n}^{2}}=I_{\rho^{2}}.$ Let $w_{n}\in B_{\rho_{n}}$ such that $I(w_{n})<I_{\rho_{n}^{2}}+\frac{1}{n}<\frac{1}{n}$. Hence, from Gagliardo--Nirenberg inequality, Sobolev inequality and Lemma \ref{L2}, we can deduce
\begin{align*}
\begin{split}
\frac{1}{n}&>I(w_{n})\\
&\ge\frac{1}{2}\left\|\nabla w_{n}\right\|_{2}^{2}-\frac{1}{p}\left\|w_{n}\right\|_{p}^{p}\\
&\ge\frac{1}{2}\left\|\nabla w_{n}\right\|_{2}^{2}-C\rho_{n}^{\frac{6-p}{2}}\left\|\nabla w_{n}\right\|_{2}^{\frac{3(p-2)}{2}}.
\end{split}
\end{align*}
Due to $\frac{3(p-2)}{2}<2$ and $\left\{\rho_{n}\right\}$ is bounded sequence, Hence we know that $\left\{w_{n}\right\}$ is bounded in $H^{1}(\mathbb{R}^{3}).$ Moreover, $\left\{A(w_{n})\right\}$ and $\left\{C(w_{n})\right\}$ are bounded sequence. We also know
$$B(w_{n})=\int_{\mathbb{R}^{3}}\phi_{w_{n}}|w_{n}|^{2}dx\le C\left\|w_{n}\right\|^{4}.$$
So $\left\{B(w_{n})\right\}$ bounded in $H^{1}(\mathbb{R}^{3}).$ From simple calculation we can obtain
\begin{align}\label{A1}
\begin{split}
I_{\rho^{2}}&\le I(\frac{\rho}{\rho_{n}}w_{n})\\&=\frac{1}{2}A(\frac{\rho}{\rho_{n}}w_{n})+\frac{1}{4}
B(\frac{\rho}{\rho_{n}}w_{n})+\frac{1}{p}C(\frac{\rho}{\rho_{n}}w_{n})\\
&=\frac{1}{2}(\frac{\rho}{\rho_{n}})^{2}A(w_{n})+\frac{1}{p}(\frac{\rho}{\rho_{n}})^{p}
+\frac{1}{4}(\frac{\rho}{\rho_{n}})^{5}\int_{\mathbb{R}^{3}}\int_{\mathbb{R}^{3}}
\frac{1-e^{-\frac{\rho_{n}}{\rho}|x-y|}}{|x-y|}w^{2}_{n}(x)w^{2}_{n}(y)dxdy\\
&\le I(w_{n})+o(1)\\
&\le I_{\rho_{n}^{2}}+o(1).
\end{split}
\end{align}
On the other hand, given a minimizing sequence $\left\{v_{n}\right\}\subset B_{\rho}$ for $I_{\rho^{2}}$ we have
\begin{align}\label{A2}
\begin{split}
I_{\rho_{n}^{2}}&\le I(\frac{\rho_{n}}{\rho}v_{n})\\
&=I(v_{n})+o(1)\\
&=I_{\rho^{2}}+o(1).
\end{split}
\end{align}
According \eqref{A1} and \eqref{A2}, we get $\lim_{n\to\infty}I_{\rho_{n}^{2}}=I_{\rho^{2}}.$
\end{proof}

\begin{lemma}
$\lim_{s\to 0}\frac{I_{s^{2}}}{s^{2}}=0$.
\end{lemma}
\begin{proof}
We let $G_{\rho^{2}}=\inf\left\{\frac{1}{2}\left\|u\right\|_{D^{1,2}}^{2}-
\frac{1}{p}\int_{\mathbb{R}^{3}}|u|^{p}dx\right\}.$ We know that $G_{\rho^{2}}/\rho^{2}\le I_{\rho^{2}}/\rho^{2}<0.$ Since $G_{\rho^{2}}/\rho^{2}\to 0$ (see Appendix A in \cite{MR2826402}) we can easily conclude.
\end{proof}

\begin{lemma}
For small $\rho$ the function $I$ satisfies (v) of proposition \ref{P4}.
\end{lemma}
\begin{proof}
From Theorem 2.1 of \cite{MR2826402}, we know that under the conditional assumptions of proposition \ref{P2}  we can get $M(\rho)$ is nonempty, where $M(\rho)$ is defined in Lemma \ref{P2}. Moreover, since $0\notin M(\rho),$ $A(u),\ B(u)$ and $C(u)$ are different from zero whenever $u\in M(\rho).$
As we all know that the solution of problem \eqref{P1} must satisfied the following Poho\v{z}ev equality which was proved in \cite{MR4129340},
\begin{align}\label{A3}
\begin{split}
0=A(u)
+5B(u)+3C(u)+\frac{1}{4}\int_{\mathbb{R}^{3}}\int_{\mathbb{R}^{3}}
e^{-|x-y|}u^{2}(x)u^{2}(y)dxdy+\frac{3\omega}{2}\int_{\mathbb{R}^{3}}u^{2}dx.
\end{split}
\end{align}
Moreover, when without the $L^{2}$-norm constraint, the energy functional of problem \ref{P1} is
\begin{align*}
J(u)=A(u)+B(u)+C(u)+\frac{\omega}{2}\int_{\mathbb{R}^{2}}u^{2}dx.
\end{align*}
If $u$ is the solution of problem \ref{P1} then $\left\langle J'(u),u\right\rangle =0$ namely,
\begin{align}\label{A4}
2A(u)+4B(u)+pC(u)+\omega\int_{\mathbb{R}^{3}}|u|^{2}dx=0.
\end{align}
From \eqref{A3} and \eqref{A4}, we claim that for any $u\in M(\rho)$ we get
\begin{align}\label{A5}
\begin{split}
0=&2A(u)+B(u)-\frac{6-3p}{2}C(u)
-\frac{1}{4}\int_{\mathbb{R}^{3}}\int_{\mathbb{R}^{3}}e^{-|x-y|}u^{2}(x)u^{2}(y)dxdy.
\end{split}
\end{align}
Now, for $u\neq 0$ we compute explicitly $h_{g_{u}}(\theta)$ by choosing the family of scaling paths of $u$ with $\beta\in\mathbb{R}$ given  by
$$\textrm{G}_{u}^{\beta}=\left\{g_{u}(\theta)=\theta^{1-\frac{3}{2}\beta}
u(x/\theta^{\beta})\right\}.$$
All the paths of this family have as associated function $\Theta(\theta)=\theta^{2}.$ We get
\begin{align}
\begin{split}
h_{g_{u}}(\theta)&=I(g_{u}(\theta))-\theta^{2}I(u)\\
&=(\theta^{2-2\beta}-\theta^{2})A(u)+
(\theta^{(1-\frac{3}{2}\beta)p+3\beta}-\theta^{2})C(u)-\theta^{2}B(u)\\
&\ \ \ +\frac{1}{4}\theta^{4-\beta}
\int_{\mathbb{R}^{3}}\int_{\mathbb{R}^{3}}\frac{1-e^{-\theta^{\beta}|x-y|}}
{|x-y|}u(x)^{2}u(y)^{2}dxdy,
\end{split}
\end{align}
which show that $h_{g_{u}}$ is differentiable for every $g_{u}\in\textrm{G}_{u}^{\beta}.$ We have also, for $g_{u}\in\textrm{G}_{u}^{\beta}:$
\begin{align}
\begin{split}
h'_{g_{u}}(\theta)=&\left((2-2\beta)\theta^{2-2\beta-1}-2\theta\right)A(u)-2\theta B(u)+\left([(1-\frac{3}{2}\beta)p+3\beta]
\theta^{(1-\frac{3}{2}\beta)p+3\beta-1}-
2\theta\right)C(u)\\
&+\frac{1}{4}(4-\beta)\theta^{4-\beta-1}
\int_{\mathbb{R}^{3}}\int_{\mathbb{R}^{3}}
\frac{1-e^{-\theta^{\beta}|x-y|}}{|x-y|}u^{2}(x)u^{2}(y)dxdy\\
&+\frac{1}{4}\theta^{4-\beta}\int_{\mathbb{R}^{3}}\int_{\mathbb{R}^{3}}
e^{-\theta^{\beta}|x-y|}\beta\theta^{\beta-1}u^{2}(x)u^{2}(y)dxdy.
\end{split}
\end{align}
Let $\theta=1$, we get
\begin{align}
\begin{split}
h'_{g_{u}}(1)=&-2\beta A(u)+\left((1-\frac{3}{2}\beta)p+3\beta-2\right)C(u)+(2-\beta)B(u)\\
&+\frac{\beta}{4}\int_{\mathbb{R}^{3}}\int_{\mathbb{R}^{3}}e^{-|x-y|}
u^{2}(x)u^{2}(y)dxdy.
\end{split}
\end{align}
Assume there exists a sequence $\left\{u_{n}\right\}\subset M(\rho)$ with $\rho\ge\left\|u_{n}\right\|_{2}=\rho_{n}\to 0$ such that for all $\beta\in\mathbb{R}.$
\begin{align}\label{A6}
\begin{split}
0&=h'_{g_{u_{n}}}(1)\\&=-2\beta A(u_{n})+\left((1-\frac{3}{2}\beta)p+3\beta-2\right)C(u_{n})+(2-\beta)B(u_{n})\\
&\ \ \ +\frac{\beta}{4}\int_{\mathbb{R}^{3}}\int_{\mathbb{R}^{3}}e^{-|x-y|}
u_{n}^{2}(x)u_{n}^{2}(y)dxdy.
\end{split}
\end{align}
From \eqref{A5} and \eqref{A6}, we can get
\begin{align}\label{A15}
2B(u_{n})+(p-2)C(u_{n})=0.
\end{align}
and hence
\begin{align}\label{A7}
\begin{split}
&B(u_{n})=A(u_{n})-\frac{1}{2}\int_{\mathbb{R}^{3}}\int_{\mathbb{R}^{3}}
e^{-|x-y|}u^{2}_{n}(x)u^{2}_{n}(y)dxdy,\\
&C(u_{n})=\frac{2}{2-p}A(u_{n})-\frac{1}{2-p}\int_{\mathbb{R}^{3}}
\int_{\mathbb{R}^{3}}e^{-|x-y|}u^{2}_{n}(x)u^{2}_{n}(y)dxdy,\\
&I(u_{n})=\frac{6-2p}{2-p}A(u_{n})-\frac{4-p}{2(2-p)}\int_{\mathbb{R}^{3}}
\int_{\mathbb{R}^{3}}e^{-|x-y|}u^{2}_{n}(x)u^{2}_{n}(y)dxdy.
\end{split}
\end{align}
Through direct calculation, we can know that $e^{-|x-y|}$ is a bounded function, then $$\int_{\mathbb{R}^{3}}\int_{\mathbb{R}^{3}}e^{-|x-y|}u^{2}_{n}(x)u^{2}_{n}(y)dxdy\to 0,$$ so
\begin{eqnarray}\label{A8}
	\left\{   \begin{array}{ll}
	I(u_{n})=I_{\rho_{n}^{2}}\to 0, \\
	A(u_{n}),B(u_{n}),C(u_{n})\to 0.\\
\end{array}
\right.	
\end{eqnarray}
Notice the following Hardy-Littlewood-Sobolev inequality
$$B(u_{n})\le C\left\|u_{n}\right\|_{\frac{12}{5}}^{4},$$
that we will frequently use.

\textbf{Case} \textbf{1}: When $2<p<\frac{12}{5}.$

Then
$$B(u_{n})\le C\left\|u_{n}\right\|_{\frac{12}{5}}^{4}\le C\left\|u_{n}\right\|_{p}^{4\alpha}\left\|u_{n}\right\|_{6}^{4(1-\alpha)},\ \ \ \alpha=\frac{3p}{2(6-p)}.$$
Thanks to \eqref{A7} and Sobolev inequality $\left\|u_{n}\right\|_{6}^{2}\le 2SA(u_{n}),$ where $S$ is the best constant for Sobolev embedding $D^{1,2}(\mathbb{R}^{3}) \hookrightarrow L^{6}(\mathbb{R}^{3})$, namely $S:=\inf_{u\in D^{1,2}(\mathbb{R}^{3})\setminus\left\{0\right\}}\dfrac{\int_{\mathbb{R}^{3}}\left| \nabla u\right| ^{2}dx}{\left(\int_{\mathbb{R}^{3}} |u|^{6}dx\right)^{\frac{1}{3}}}.$ We get
$$B(u_{n})\le CB(u_{n})^{\frac{4\alpha}{p}}\left(B(u_{n})+\frac{1}{2}
\int_{\mathbb{R}^{3}}\int_{\mathbb{R}^{3}}e^{-|x-y|}u^{2}_{n}(x)u^{2}_{n}(y)dxdy\right)
^{\frac{4(1-\alpha)}{2}}.$$
Due to $\frac{4\alpha}{p}+\frac{4(1-\alpha)}{2}=3-\frac{2p}{6-p}>1$ for $2<p<\frac{12}{5},$ hence this is a contradiction with \eqref{A8}.

\textbf{Case} \textbf{2}: When $p=\frac{12}{5}.$

Due to \eqref{A15} we obtain
$$\left\|u_{n}\right\|_{\frac{12}{5}}^{{\frac{12}{5}}}=CB(u_{n})\le C\left\|u_{n}\right\|_{\frac{12}{5}}^{4},$$
which contradicts \eqref{A8}.

\textbf{Case} \textbf{3}: When $\frac{12}{5}<p<\frac{8}{3}.$

Interpolating $L^{\frac{12}{5}}$ between $L^{2}$ and $L^{p}$ we get
$$\left\|u_{n}\right\|_{p}^{p}=CB(u_{n})\le C\left\|u_{n}\right\|_{\frac{12}{5}}^{4}\le C\left\|u_{n}\right\|_{2}^{4\alpha}
\left\|u_{n}\right\|_{p}^{4(1-\alpha)},\ \ \ \ \alpha=\frac{5p-12}{6(p-2)}.$$
Since $p<4(1-\alpha)$, i.e. $p<\frac{8}{3},$ we get a contradiction with \eqref{A8}.

\textbf{Case} \textbf{4}: When $p=\frac{8}{3}.$

Again by interpolation we get
$$B(u_{n})\le C\left\|u_{n}\right\|_{\frac{12}{5}}^{4}\le C\rho_{n}^{\frac{4}{3}}\left\|u_{n}\right\|_{\frac{8}{3}}^{\frac{8}{3}},$$
and again, using that $B(u_{n})=C\left\|u_{n}\right\|_{\frac{8}{3}}^{\frac{8}{3}}$ we get a contradiction.

\textbf{Case} \textbf{5}: When $\frac{8}{3}<p<3.$

We first compute
\begin{align}
\begin{split}
\frac{I(g_{u}(\theta))}{\theta^{2}\left\|u\right\|_{2}^{2}}&=\frac{h_{g_{u}}(\theta)}
{\theta^{2}\left\|u\right\|_{2}^{2}}+\frac{I(u)}{\left\|u\right\|_{2}^{2}}\\
&=\frac{1}{\left\|u\right\|_{2}^{2}}\bigg\{\theta^{-2\beta}A(u)
+\frac{1}{4}\theta^{2-\beta}\int_{\mathbb{R}^{3}}\int_{\mathbb{R}^{3}}
\frac{1-e^{-\theta^{\beta}|x-y|}}{|x-y|}u^{2}(x)u^{2}(y)dxdy\\
&\ \ \ +
\theta^{(1-\frac{3}{2}\beta)p+3\beta-2}C(u)\bigg\}.
\end{split}
\end{align}
In this case for $u_{0}$ satisfying \eqref{A7} with $\left\|u_{0}\right\|_{2}=\rho_{0},$ hence we can get
\begin{align}
\begin{split}
\frac{I_{\theta^{2}\rho^{2}}}{\theta^{2}\rho^{2}}&\le \frac{I(g_{u}(\theta))}{\theta^{2}\rho^{2}}\\
&\le\frac{1}{\rho_{0}^{2}}\left(\theta^{-2\beta}A(u_{0})+\frac{2}{2-p}
\theta^{(1-\frac{3}{2}\beta)p+3\beta-2}A(u_{0})+\frac{1}{4}\theta^{2-\beta}
\int_{\mathbb{R}^{3}}\int_{\mathbb{R}^{3}}\frac{1}{|x-y|}u(x)^{2}u(y)^{2}dxdy\right)\\
&\le\frac{1}{\rho_{0}^{2}}\left(
\theta^{-2\beta}A(u_{0})+\theta^{2-\beta}A(u_{0})+\frac{2}{2-p}
\theta^{(1-\frac{3}{2}\beta)p+3\beta-2}A(u_{0})\right).
\end{split}
\end{align}
Now let us choose $\beta=\frac{2(2-p)}{10-3p}$ so that $0<-2\beta=(1-\frac{3}{2}\beta)p+\beta-2<2-\beta.$ Hence we can obtain
\begin{align}
\begin{split}
\frac{I_{\theta^{2}\rho_{0}^{2}}}{\theta^{2}\rho_{0}^{2}}
&\le\frac{I(g_{u_{0}}(\theta))}{\theta^{2}\rho_{0}^{2}}=\frac{A(u_{0})}{\rho_{0}^{2}}
\left[\frac{4-p}{(2-p)}\theta^{\frac{4(p-2)}{10-3p}}
+\theta^{\frac{4(4-p)}{10-3p}}\right]\\
&=\frac{\frac{4-p}{8(3-p)}
\int_{\mathbb{R}^{3}}\int_{\mathbb{R}^{3}}e^{-|x-y|}
u_{0}(x)^{2}u_{0}(y)^{2}dxdy}{\rho_{0}^{2}}\left[\frac{4-p}{2-p}\theta^{\frac{4(p-2)}
{10-3p}}+\theta^\frac{4(4-p)}{10-3p}\right]\\&
\ \ \ +\frac{\frac{2-p}{3-p}I(u_{0})}{\rho_{0}^{2}}\left[\frac{4-p}{2-p}\theta^{\frac{4(p-2)}
{10-3p}}+\theta^\frac{4(4-p)}{10-3p}\right].
\end{split}
\end{align}
Due to $\int_{\mathbb{R}^{3}}\int_{\mathbb{R}^{3}}e^{-|x-y|}u_{0}(x)^{2}u_{0}(y)^{2}dxdy\to 0$ as $s$ is enough small and $\theta^{2}\rho_{0}^{2}=s^{2},$ we get
\begin{align}\label{A11}
\frac{I_{s^{2}}}{s^{2}}\le -cs^{\frac{4(p-2)}{10-3p}}+o(s^{\frac{4(p-2)}{10-3p}}).
\end{align}
On the other hand for $u_{n}$ satisfying \eqref{A7}
$$\left\|u_{n}\right\|_{p}^{p}=CB(u_{n})\le C\left\|u_{n}\right\|_{\frac{12}{5}}^{4}
\le C\left\|u_{n}\right\|_{2}^{4\alpha}\left\|u_{n}\right\|_{p)}^{4(1-\alpha)},\ \ \ \ \alpha=\frac{5p-12}{6(p-2)},$$
that is
\begin{align}\label{A9}
\left\|u_{n}\right\|_{p}^{p}\le C\rho_{n}^{4\alpha}\left\|u_{n}\right\|_{p}^{4(1-\alpha)}.
\end{align}
we know that $\frac{8}{3}<p$ from \eqref{A9} we get $\left\|u_{n}\right\|_{p}^{p}\le C\rho_{n}^{\frac{4(p-2)}{3p-8}}.$ Combine with \eqref{A7} we obtain
\begin{align}\label{A10}
\frac{I_{\rho_{n}^{2}}}{\rho_{n}^{2}}\ge -C\rho_{n}^{\frac{4(p-2)}{3p-8}}.
\end{align}
Combine \eqref{A11} with \eqref{A10} we obtain
$$-C\rho_{n}^{\frac{4(p-2)}{3p-8}}\le \frac{I_{\rho_{n}^{2}}}{\rho_{n}^{2}}\le
-cs^{\frac{4(p-2)}{10-3p}}+o(s^{\frac{4(p-2)}{10-3p}}).$$
Since $\frac{4(p-2)}{3p-8}>\frac{4(p-2)}{10-3p},$ for $\rho_{n}\to 0$ we get a contradiction.
\end{proof}
\section{The orbital stability}
Before proving Theorem \ref{T3}, the definition of orbital stability is introduced. Let
$$S_{\rho}=\left\{e^{i\theta}u(x):\theta\in[0,2\pi),\left\|u\right\|_{2}=\rho,I(u)=I_{\rho^{2}}\right\}.$$
We call $S_{\rho}$ is orbitally stable if for every $\varepsilon>0$ there exists $\delta>0$ such that for any $\psi_{0}\in H^{1}(\mathbb{R}^{3})$ with $\inf_{v\in S_{\rho}}\left\|v-\psi_{0}\right\|_{H^{1}(\mathbb{R}^{3};\mathbb{C})}<\delta$ we have $$\forall t>0\ \ \ \ \inf_{v\in S_{\rho}}\left\|\psi(t,\cdot)-v\right\|_{H^{1}(\mathbb{R}^{3};\mathbb{C})}<\varepsilon,$$
where $\psi(t,\cdot)$ is the solution of problem \eqref{P1} with initial datum $\psi_{0}.$ It is worth noting that $S_{\rho}$ is translational invariant, that is, when $\upsilon\in S_{\rho}$ then $\upsilon(\cdot -y)\in S_{\rho}$ for every $y\in\mathbb{R}^{3}.$
\begin{proof}[\textbf{Proof of Theorem 1.3}]
Using the method of contradiction, it is assumed that there exists a $\rho>0$ such that $S_{\rho}$ is not orbitally stable. Therefore, there exist $\varepsilon>0$ and a sequence of initial data $\left\{\psi_{n},0\right\}\subset H^{1}(\mathbb{R}^{3})$ and $\left\{t_{n}\right\}\subset\mathbb{R},$ such that the maximal solution $\psi_{n}$, which is global and $\psi_{n}(0,\cdot)=\psi_{n,0},$ satisfies
$$\lim_{n\to +\infty}\inf_{v\in S_{\rho}}\left\|\psi_{n,0}-v\right\|_{H^{1}(\mathbb{R}^{3})}=0,$$
and
$$\inf_{v\in S_{\rho}}\left\|\psi_{n}(t_{n},\cdot)-v\right\|_{H^{1}(\mathbb{R}^{3})}\ge\varepsilon.$$
Therefore, there exists $u_{\rho}\in H^{1}(\mathbb{R}^{3})$ which is a minimizer of $I_{\rho^{2}}$ and $\theta\in\mathbb{R}$ such that $v=e^{i\theta}u_{\rho}$ and
$$\left\|\psi_{n,0}\right\|_{2}\to\left\|v\right\|_{2}=\rho\ \ \ \ \ I(\psi_{n,0})\to I(v)=I_{\rho^{2}}.$$
On the other hand, we can assume that $\psi_{n,0}\in B_{\rho}.$ Then $\left\{\psi_{n,0}\right\}$ is a minimizing sequence for $I_{\rho^{2}},$ and due to
$$I(\psi_{n}(\cdot,t_{n}))=I(\psi_{n,0}).$$
So $\left\{\psi_{n}(\cdot,t_{n})\right\}$ is a minimizing sequence for $I_{\rho^{2}}.$ Since we have that every minimizing sequence has a subsequence converging in $H^{1}(\mathbb{R}^{3})$ to a minimum on the sphere $B_{\rho},$ we have a contradiction.
\end{proof}

\end{document}